\documentclass[10pt,twoside]{amsart}
\usepackage{amsfonts}
\usepackage{graphicx}
\usepackage{amssymb}
\usepackage{amsmath}
\usepackage{amsthm}
\usepackage{tikz}
\usepackage{amscd}
\usepackage[all]{xy}
\usepackage{dsfont}
\usepackage{indentfirst}
\usepackage{url}
\usepackage[document]{ragged2e}

\newtheorem{thm}{Theorem}

\newtheorem{lem}{Lemma}

\theoremstyle{definition}

\theoremstyle{remark}

\numberwithin{equation}{section}

\begin{document}

\title{On the Elementary Proof of the Inverse ErdŐs-Heilbronn Problem}

\author{Shengning Zhang}

\address{}

\email{}

\begin{abstract}
In this article, we studied the inverse Erdős-Heilbronn problem with the restricted sumset from two components $A$ and $B$ that are not necessarily the same. We give a completely elementary proof for the problem in $\mathbb{Z}$ and some partial results that contributes to the elementary proof of the problem in $\mathbb{Z}/p\mathbb{Z}$, avoiding the usage of the powerful polynomial method and the Combinatorial Nullstellensatz.

\end{abstract}

\maketitle

\section{Introduction}
\setlength{\parindent}{0pt}
A basic object in Additive Number Theory is the sumset of sets $A$ and $B$: 
$$A+B=\{a+b: a\in A, b\in B \}.$$ A simple problem in Additive Number Theory is: Given two subsets $A$ and $B$ of sets of integers, what properties can we determined about $A+B$? A classic result is the Cauchy-Davenport Theorem proved by Cauchy \cite{Ca1813} in 1813 and independently by Davenport \cite{Da1935} in 1935. Let $p$ be a prime, if $A$ and $B$ are nonempty subsets of $\mathbb{Z}/p\mathbb{Z}$, then the theorem asserts that $$\left | A+B \right |\ge \min\{p, |A|+|B|-1\}.$$
Similarly, we can define the restricted sumsets of sets $A$ and $B$: $$A\widehat{+}B=\{a+b: a\in A, b\in B, a\ne b \}.$$ In 1964, P. Erdős and H. Heilbronn \cite{Eh1964} conjectured that $$\left | A\widehat{+}A \right |\ge \min\{p, 2|A|-3\},$$ where $p$ is a prime and $A$ is a nonempty subset of $\mathbb{Z}/p\mathbb{Z}$. The conjecture was then proved by J. A. Dias da Silva and Y. O. Hamidoune \cite{Jy1994} in 1994 through the usage of methods from linear algebra. The more general case $$\left | A\widehat{+}B \right |\ge \min\{p, |A|+|B|-3\}$$ was established and proved by N. Alon, Melvin B. Nathanson, and Imre Z. Ruzsa \cite{Na1995} in 1995 using the polynomial method. 
\vspace{1pc}

These theorems above depict the properties of sumsets and restricted sumsets given the knowledge of the individual sets that make up the sumsets or restricted sumsets. Thus, one may consider the inverse directions of these problems and there indeed exists several beautiful results. In particular, Gy. Károlyi \cite{Gy2009} points out the following inverse problem of the Erdős-Heilbronn conjecture:
\vspace{1pc}

Let sets $A, B\subseteq \mathbb{Z}/p\mathbb{Z}$ be nonempty sets such that $p\ge |A|+|B|-2$, where $p$ is a prime. Then $|A\widehat{+}B|=|A|+|B|-3$ if and only if $A=B$ and one of the following holds:
\begin{enumerate}
    \item $|A|=2$ or $|A|=3$;
    \item $|A|=4$ and $A=\{a, a+d, c, c+d\}$;
    \item $|A|\ge 5$ and $A$ is an arithmetic progression.
\end{enumerate}
\vspace{1pc}

This paper is motivated to give an elementary proof of the results. Firstly, the original proof of this statement \cite{Gy2009} relies on the combinatorial Nullstellensatz and the polynomial method through the entire proof of the theorem. Inspired by \cite{Su2013}, we believe that this should be unnecessary for such an elementary statement. Secondly, we hope to obtain a more explicit view to the structures of the sumsets in $\mathbb{Z}$ and $\mathbb{Z}/p\mathbb{Z}$ by an elementary proof. 
\vspace{1pc}

Our strategy is to first prove the inverse problem of the Erdős-Heilbronn conjecture for $\mathbb{Z}$, which is of independent interests. Our result generalizes Nathanson's proof \cite{GTM165}, where the case $A=B$ is studied. In our work $A$ and $B$ are not necessarily the same. We then give some partial results for the elementary proof of the case in $\mathbb{Z}/p\mathbb{Z}$ through Theorem 5, Theorem 6 and Theorem 7.

\section{Notations}
To simplify our proofs, we will use the following notations throughout the rest of the paper:
\begin{enumerate}
    \item Let $G$ be the group $\mathbb{Z}$ or $\mathbb{Z}/p\mathbb{Z}$. For nonempty finite sets $A$, $B\subseteq G$ with $|A|$, $|B|\ge 2$, note $$A=\{a_1, a_2,\cdots, a_m\}\ (m=|A|)$$ and $$B=\{b_1, b_2,\cdots, b_n\}\ (n=|B|).$$ If $G=\mathbb{Z}$, we always assume that $a_1<a_2<\cdots<a_m$ and $b_1<b_2<\cdots<b_n$;\ If $G=\mathbb{Z}/p\mathbb{Z}$, we identity it as set with $\{0,1,2,\cdots, p-1\}$, and without the abuse of notation, we also assume  $a_1<a_2<\cdots<a_m$ and $b_1<b_2<\cdots<b_n$;
    \vspace{1pc}
    
    \item For nonempty finite sets $A$, $B\subseteq G$ with $|A|, |B|\ge 2$, define the restricted sumset of $A$ and $B$ as $A\widehat{+}B$. $A\widehat{+}B=\{a+b: a\in A, b\in B, a\ne b \}$ if $G=\mathbb{Z}$ and $A\widehat{+}B=\{a+b \mod{p}: a\in A, b\in B, a\ne b \}$ if $G=\mathbb{Z}/p\mathbb{Z}$;
    
    \vspace{1pc}
    
    \item For nonempty finite sets $A$, $B\subseteq G$ with $|A|, |B|\ge 2$, if $$|A\widehat{+}B|=|A|+|B|-3,$$ then we call the unordered pair $(A, B)$ a critical pair;
    \vspace{1pc}
    
    \item For nonempty finite set $A\subseteq G$ with $|A|\ge 2$, if $A$ can be written into the form $A=\{\tau+id: i=0, 1,\cdots, |A|-1\}$ for some $\tau, d\in G$ with $d\ne 0$, then we call set $A$ a standard set;
    \vspace{1pc}
    
    \item If sets $A$ and $B$ are both standard sets and $A=B$, then we call the unordered pair $(A, B)$ a standard pair.
     \vspace{1pc}
     
     \item Suppose $m, n\in\mathbb{Z}_+\cup\{0\}$ and $p$ is a prime. Let $X\in \mathbb{Z}/p\mathbb{Z}$ be a nonempty set with $|X|\ge 2$. Suppose $d$ is a given generator of $\mathbb{Z}/p\mathbb{Z}$. Then $[m, n]_d$ is called a gap with respect to $d$ in $X$ if $rd\notin X$ for $\forall r\in\{m, m+1,\cdots, n\}$. The length of the gap $[m, n]_d$ is $(n-m+1)$.
\end{enumerate}

\section{The Cases of $\mathbb{Z}$ for $2\le|A|\le 4$}
\begin{thm}
Let finite sets $A, B\subseteq\mathbb{Z}$ be nonempty sets. If $|A|=2$, then $(A, B)$ is a critical pair if and only if $A=B$.
\end{thm}
\begin{proof}
We first consider the sufficiency of the conditions. If $A=B$, then assume that $A=B=\{a_1, a_2\}$. Thus, $$|A\widehat{+}B|=|\{a_1+a_2\}|=1=|A|+|B|-3.$$
\vspace{1pc}

For the necessity of the conditions, suppose $A=\{a_1, a_2\}$. If $a_2\ne b_n$, due to $a_2>a_1$ and $b_n=\max_{x\in B}\{x\}$, we have $$a_2+b_n\in (A\widehat{+}B)\setminus(\{a_1\}\widehat{+}B).$$ Together with $$(\{a_1\}\widehat{+}B)\cup \{a_2+b_n\}\subseteq A\widehat{+}B,$$ there is  $$|A\widehat{+}B|\ge|(\{a_1\}\widehat{+}B)\cup \{a_2+b_n\}|=|B|-1+1>|A|+|B|-3.$$ A contradiction! Therefore, $a_2= b_n$. 
\vspace{1pc}

Similarly, if $a_1\ne b_1$, we have $$a_1+b_1\in (A\widehat{+}B)\setminus(\{a_2\}\widehat{+}B),$$ because $a_1<a_2$ and $b_1=\min_{x\in B}\{x\}$. Thus, combining with $$(\{a_2\}\widehat{+}B)\cup \{a_1+b_1\}\subseteq A\widehat{+}B,$$ there is $$|A\widehat{+}B|\ge|(\{a_2\}\widehat{+}B)\cup \{a_1+b_1\}|=|B|-1+1>|A|+|B|-3.$$ Consequently, $a_1=b_1$. 
\vspace{1pc}

Since $a_1$, $a_2\in B$, $|\{a_1\}\widehat{+}B|=|\{a_2\}\widehat{+}B|=|B|-1=|A|+|B|-3=|A\widehat{+}B|.$ Thus, due to  $\{a_1\}\widehat{+}B$, $\{a_2\}\widehat{+}B\subseteq A\widehat{+}B,$ there is $\{a_1\}\widehat{+}B=\{a_2\}\widehat{+}B=A\widehat{+}B.$ Now, because $a_1\ne b_2$, we have $a_1+b_2\in \{a_1\}\widehat{+}B=\{a_2\}\widehat{+}B.$ Therefore, $\exists b_i\in B\ (i\in \mathbb{Z}_+)$, such that $a_1+b_2=a_2+b_i$. Notice that $a_2>a_1$, so we must have $b_i<b_2$. According to the definition of $B$, $b_i=b_1$. Together with $a_1=b_1$, we have $b_n=a_2=b_2.$ So, we have proven that $A=B$.

\end{proof}
\begin{thm}
Let finite sets $A, B\subseteq\mathbb{Z}$ be nonempty sets. If $|A|=3$, then $(A, B)$ is a critical pair if and only if $A=B$.
\end{thm}
\begin{proof}
Considering the sufficiency of the conditions, if $A=B$, we may assume that $A=B=\{a_1, a_2, a_3\}$, then$$|A\widehat{+}B|=|\{a_1+a_2, a_1+a_3, a_2+a_3\}|=3=|A|+|B|-3.$$
\vspace{1pc}

In order to prove the necessity of the conditions, we will first prove that $A\subseteq B$. If $a_1\notin B$, then $$|\{a_1\}\widehat{+}B|=|B|=|A|+|B|-3=|A\widehat{+}B|.$$ Since $\{a_1\}\widehat{+}B\subseteq A\widehat{+}B$ and $\{a_2, a_3\}\widehat{+}B\subseteq A\widehat{+}B$, we have $$\{a_2, a_3\}\widehat{+}B\subseteq\{a_1\}\widehat{+}B=A\widehat{+}B.$$ This indicates that for $\forall b_i\in B\ (i\in \mathbb{Z}_+)$ with $b_i\ne a_2$, $\exists b_j\in B\ (j\in \mathbb{Z}_+)$ such that $a_2+b_i=a_1+b_j.$ Suppose $a_2\ne b_n$, then $a_2+b_n=a_1+b_j$. However, for $\forall b_k\in B\ (k\in\mathbb{Z}_+)$, there is $a_2+b_n>a_1+b_k$. A contradiction! Thus, we must have $a_2=b_n$. Notice that for $\forall b_k\in B\ (k\in\mathbb{Z}_+$), there is $a_3+b_n>a_1+b_k$. Therefore, by applying the same argument, there is $a_3=b_n$, which contradicts to $a_3>a_2$. Consequently, we must have $a_1\in B$.
\vspace{1pc}

If $a_3\notin B$, then $$|\{a_3\}\widehat{+}B|=|B|=|A|+|B|-3=|A\widehat{+}B|.$$ Since $\{a_3\}\widehat{+}B\subseteq A\widehat{+}B$ and $\{a_1, a_2\}\widehat{+}B\subseteq A\widehat{+}B$, we have $$\{a_1, a_2\}\widehat{+}B\subseteq\{a_3\}\widehat{+}B=A\widehat{+}B.$$ This indicates that for $\forall b_i\in B\ (i\in \mathbb{Z}_+)$ with $b_i\ne a_2$, $\exists b_j\in B\ (j\in \mathbb{Z}_+)$ such that $a_2+b_i=a_3+b_j.$ Suppose $a_2\ne b_1$, then $a_2+b_1=a_3+b_j$. However, we know that for $\forall b_k\in B\ (k\in\mathbb{Z}_+)$, there is $a_2+b_1<a_3+b_k$. A contradiction! Thus, we must have $a_2=b_1$. Notice that for $\forall b_k\in B\ (k\in\mathbb{Z}_+$), there is $a_1+b_1<a_2+b_k$. Therefore, by applying the same argument, there is $a_1=b_1$, which contradicts to $a_2>a_1$. Consequently, we must have $a_3\in B$.
\vspace{1pc}

Suppose $a_1\ne b_1$, then $a_1+b_1\in (A\widehat{+}B)\setminus(\{a_3\}\widehat{+}B)$, for $a_1<a_3$ and $b_1=\min_{x\in B}\{x\}$. Thus, $a_2=b_1$, or else we can get $a_2+b_1\in (A\widehat{+}B)\setminus(\{a_3\}\widehat{+}B)$ with the same argument above. Apparently, $(\{a_3\}\widehat{+}B)\cup \{a_1+b_1, a_2+b_1\}\subseteq A\widehat{+}B.$ However, because $(A, B)$ is a critical pair, \begin{align*}
 |(\{a_3\}\widehat{+}B)\cup \{a_1+b_1, a_2+b_1\}| &= |\{a_1\}\widehat{+}B|+|\{a_1+b_1, a_2+b_1\}| \\
 &= (|B|-1)+2 \\
 &=  |A|+|B|-2 \\
 &=  |A\widehat{+}B|+1
\end{align*}
leads to a contradiction! Therefore, $a_1=b_1$.
\vspace{1pc}

Suppose $a_3\ne b_n$, then $a_3+b_n\in A\widehat{+}B\setminus\{a_1\}\widehat{+}B$, for $a_1<a_3$ and $b_n=\max_{x\in B}\{x\}$. So, $a_2=b_n$, or else we can get $a_2+b_n\in (A\widehat{+}B)\setminus(\{a_1\}\widehat{+}B)$ with the same argument above. Apparently, $(\{a_1\}\widehat{+}B)\cup \{a_2+b_n, a_3+b_n\}\subseteq A\widehat{+}B.$ However, because $(A, B)$ is a critical pair, 
\begin{align*}
 |(\{a_1\}\widehat{+}B)\cup \{a_2+b_n, a_3+b_n\}| &= |\{a_1\}\widehat{+}B|+|\{a_2+b_n, a_3+b_n\}| \\
 &= (|B|-1)+2 \\
 &=  |A|+|B|-2 \\
 &=  |A\widehat{+}B|+1
\end{align*}
leads to a contradiction! Therefore, $a_3=b_n$.
\vspace{1pc}

If $a_2\notin B$, then $$|\{a_2\}\widehat{+}B|=|B|=|A|+|B|-3=|A\widehat{+}B|.$$ Since $\{a_2\}\widehat{+}B\subseteq A\widehat{+}B$ and $\{a_1, a_3\}\widehat{+}B\subseteq A\widehat{+}B$, we have $$\{a_1, a_3\}\widehat{+}B\subseteq\{a_2\}\widehat{+}B=A\widehat{+}B.$$ Through Theorem 2 we know that $|B|\ge 3$, for $|B|=2$ indicates that $|A|=2$. Thus, $b_{n-1}\notin \{b_1, b_n\}$. Using this condition, we have $a_3+b_{n-1}\in A\widehat{+}B=\{a_2\}\widehat{+}B.$ This implies that $\exists b_i\in B\ (i\in \mathbb{Z}_+)$ such that $a_3+b_{n-1}=a_2+b_j.$ Since $a_3>a_2$, there must be $b_j>b_{n-1}.$ Thus, $b_j=b_n$. Together with $a_3=b_n$, we have $a_2=b_{n-1}\in B.$ A contradiction! Therefore, $a_2\in B$. Combining with the arguments above, we have $A\subseteq B$. 
\vspace{1pc}

Since $a_1\in B$, we have $|\{a_1\}\widehat{+}B|=|B|-1$. Due to $b_{n-1}\notin \{b_1, b_n\}=\{a_1, a_3\}$, $a_1<b_{n-1}$. So, $a_1+b_n<a_3+b_{n-1}$. Therefore, $a_3+b_{n-1}\in (A\widehat{+}B)\setminus(\{a_1\}\widehat{+}B)$. Consequently, $$|\{a_3+b_{n-1}\}\cup(\{a_1\}\widehat{+}B)|=|B|=|A|+|B|-3=|A\widehat{+}B|.$$
Since $(\{a_3+b_{n-1}\})\cup(\{a_1\}\widehat{+}B)\subseteq A\widehat{+}B$, we have $\{a_3+b_{n-1}\}\cup(\{a_1\}\widehat{+}B)=A\widehat{+}B.$ Apparently, $a_2+b_n\in (A\widehat{+}B)\setminus (\{a_1\}\widehat{+}B)$. With the same argument, there is $\{a_2+b_{n}\}\cup(\{a_1\}\widehat{+}B)=A\widehat{+}B.$ This together with the previous argument above leads to $$\{a_3+b_{n-1}\}\cup(\{a_1\}\widehat{+}B) =\{a_2+b_{n}\}\cup(\{a_1\}\widehat{+}B).$$ Thus, $a_3+b_{n-1}=a_2+b_{n},$ indicating that $a_2=b_{n-1}$.
\vspace{1pc}

Suppose $|B|\ge 4$, then $\exists b_{n-2}\in B\setminus\{b_1, b_{n-1}, b_n\}$. If $a_3+b_{n-2}\in \{a_1\}\widehat{+}B$, then $\exists b_i\in B\setminus\{a_1\}\ (i\in \mathbb{Z}_+)$ such that $a_1+b_i=a_3+b_{n-2}.$ Since $a_1<a_3$, we must have $b_i>b_{n-2}.$ Thus, $b_i=b_n$. This leads to $a_2=b_{n-2}$. A contradiction! Therefore, $a_3+b_{n-2}\notin \{a_1\}\widehat{+}B.$ Combining $$\{a_3+b_{n-2}\}\cup (\{a_1\}\widehat{+}B)\subseteq A\widehat{+}B$$ with $$|\{a_3+b_{n-2}\}\cup (\{a_1\}\widehat{+}B)|=|B|=|A\widehat{+}B|,$$ we have $\{a_3+b_{n-2}\}\cup (\{a_1\}\widehat{+}B)=A\widehat{+}B.$ Recalling that $$\{a_2+b_{n}\}\cup(\{a_1\}\widehat{+}B)=A\widehat{+}B,$$ we can get $a_2+b_{n}=a_3+b_{n-2}$. Thus, $a_2=b_{n-2}$, which is not possible! Therefore, $|B|=3$. This implies that $A=B$.

\end{proof}
\begin{thm}
Let finite sets $A, B\subseteq\mathbb{Z}$ be nonempty sets. If $|A|=4$, then $(A, B)$ is a critical pair if and only if $A=B$ and
$A=\{a, c, a+d, c+d\}.$
\end{thm}
\begin{proof}
We first consider the sufficiency of the conditions. If $A=B$ and
$A=\{a, a+d, c, c+d\}$, then we have $$|A\widehat{+}B|=|\{a+c, 2a+c, 2c+d, a+c+d, a+c+2d\}|=5=|A|+|B|-3.$$ 
\vspace{1pc}

Now, consider the necessity of conditions. Apparently, $$|(\{a_4\}\widehat{+}B)\cup (\{b_1\}\widehat{+}A)|=|\{a_4\}\widehat{+}B|+|\{b_1\}\widehat{+}A|-|(\{a_4\}\widehat{+}B)\cap (\{b_1\}\widehat{+}A)|.$$ Noticing that $(\{a_4\}\widehat{+}B)\cap (\{b_1\}\widehat{+}A)=\{a_4+b_1\},$ we have $$|(\{a_4\}\widehat{+}B)\cup (\{b_1\}\widehat{+}A)|\ge (|A|-1)+(|B|-1)-1=|A\widehat{+}B|.$$ Since $(\{a_4\}\widehat{+}B)\cup (\{b_1\}\widehat{+}A)\subseteq A\widehat{+}B$, there is $(\{a_4\}\widehat{+}B)\cup (\{b_1\}\widehat{+}A)=A\widehat{+}B$. Thus, $|\{a_4\}\widehat{+}B|=|B|-1$ and $|\{b_1\}\widehat{+}A|=|A|-1$. This gives us $a_4\in B$ and $b_1\in A$. Using the same argument, we have $(\{a_1\}\widehat{+}B)\cup (\{b_n\}\widehat{+}A)=A\widehat{+}B$, $a_1\in B$ and $b_n\in A.$ Through the definition of $a_1, a_4, b_1$ and $b_n$, there must be $a_1=b_1$ and $a_4=b_n$. Then, there are 
\begin{align*}
 A\widehat{+}B &= (\{a_4\}\widehat{+}B)\cup (\{b_1\}\widehat{+}A) \\
 &= \{a_4+b_1, a_4+b_2, \cdots, a_4+b_{n-1}\}\cup \{b_1+a_2, b_1+a_3\} \\
 &= \{a_4+b_1, a_4+b_2, \cdots, a_4+b_{n-1}\}\cup \{a_1+a_2, a_1+a_3\}
\end{align*}
and
\begin{align*}
 A\widehat{+}B &= (\{a_1\}\widehat{+}B)\cup (\{b_n\}\widehat{+}A) \\
 &= \{a_1+b_2, a_1+b_3, \cdots, a_1+b_{n}\}\cup \{b_n+a_2, b_n+a_3\}. \\
\end{align*}
Rearrange the elements in $A\widehat{+}B$, and then we have 
\begin{align*}
 A\widehat{+}B &= \{a_1+a_2, a_1+a_3, a_4+b_1, a_4+b_2, \cdots, a_4+b_{n-1}\} \\
 &= \{a_1+b_2, a_1+b_3, \cdots, a_1+b_{n}, a_2+b_n, a_3+b_n\}.
\end{align*}
Notice that $a_1+a_2<a_1+a_3<\cdots<a_4+b_{n-1}$ and $a_1+b_2<a_1+b_3<\cdots<a_3+b_n$, then there must be $a_4+b_{n-1}=a_3+b_n$ and $a_1+a_2=a_1+b_2$. So, we can get $a_2=b_2$, which further indicates that $a_3=b_3$. What's more, $a_4=b_n$ shows that $a_3=b_{n-1}$. As a result, $b_3=b_{n-1}$. Therefore, $$ A=\{a_1, a_2, a_3, a_4\}=\{b_1, b_2, b_3, b_4\}=B.$$
It's now easy for us to find that $A\widehat{+}B=\{a_1+a_4, a_2+a_4, a_3+a_4, a_1+a_2, a_1+a_3\}.$ Since we also have $a_2+a_3\in A\widehat{+}B$, it's apparent that $a_2+a_3=a_1+a_4$. This means $a_3-a_1=a_4-a_2$. By letting $a_3-a_1=d$, $a_1=a$ and $a_2=c$, we have $A=\{a, c, a+d, c+d\}$, which completes the proof.

\end{proof}

\section{The Cases of $\mathbb{Z}$ for $|A|\ge 5$}
\begin{thm}
Let sets $A, B\subseteq\mathbb{Z}$ be nonempty sets. If $|A|\ge 5$, then $(A, B)$ is a critical pair if and only if $(A, B)$ is a standard pair.
\end{thm}

\begin{proof}
We still first consider the sufficiency of the conditions. Since $(A, B)$ is a standard pair, we may assume that $$A=B=\{a_1, a_2,\cdots, a_m\}=\{a+id: i=0, 1,\cdots, m-1\}$$ for some $a, d\in \mathbb{Z}_{+}$. Define $A'=B'=\{id: i=0, 1, \cdots, m-1\}$. For $\forall x+y\in A\widehat{+}B$, there is $(x-a)+(y-a)\in A'\widehat{+}B'$. Similarly, for $\forall u+w\in A'\widehat{+}B'$, there is $(u+a)+(w+a)\in A\widehat{+}B$. So, $|A\widehat{+}B|=|A'\widehat{+}B'|$. Now, according to the definition of sumset and restricted sumset, we have $$A'\widehat{+}B'\subseteq A'+B'$$ and $$|A'+B'|=|\{id: i=0, 1,\cdots, 2m-2\}|=|A'|+|B'|-1\ge |A'\widehat{+}B'|. $$ Apparently, there are $0\notin A'\widehat{+}B'$ and $(2m-2)d\notin A'\widehat{+}B'$. Thus, $A'\widehat{+}B'\subseteq \{id: i=1,\cdots, 2m-3\}$. For $\forall x\in \left[1, m-1 \right ] \cap\mathbb{Z}$, we can choose $0\in A'$ and $xd\in B'$. Since $0\ne xd$, we have $xd=0+xd\in A'\widehat{+}B'$. For  $\forall x\in \left[m, 2m-3 \right ] \cap\mathbb{Z}$, we can choose $(m-1)d\in A'$ and $(x-m+1)d\in B'$. Since $x\in \left[m, 2m-3 \right ] \cap\mathbb{Z}$, $$(x-m+1)d<(m-1)d.$$ Therefore, we have $xd=(m-1)d+(x-m+1)d\in A'\widehat{+}B'$. So, for $\forall x\in \left[1, 2m-3 \right ]$, $xd\in A'\widehat{+}B'.$ Combining with $A'\widehat{+}B'\subseteq \{id: i=1,\cdots, 2m-3\}$, we have $A'\widehat{+}B'=\{id: i=1,\cdots, 2m-3\}$. As a result, $$|A\widehat{+}B|=|A'\widehat{+}B'|=|\{id: i=1,\cdots, 2m-3\}|=2m-3=|A|+|B|-3.$$
\vspace{1pc}

To prove the necessity of the conditions, let's consider the following lemmas. 

\begin{lem}
If finite sets $A, B\subseteq\mathbb{Z}$ are nonempty sets, then $|A+B|\ge |A|+|B|-1.$

\end{lem}
\begin{proof}
 Since $A$ and $B$ are finite sets, define the element with the largest absolute value in $A$ as $a_x$ and the element with the largest absolute value in $B$ as $b_y$. Apparently, $\exists p$ such that $p$ is a prime number and  $p>\max\{a_m, b_n, |a_x|+|b_y| , |A|+|B|-1\}$. Thus, applying the Cauchy-Davenport theorem in $\mathbb{Z}/p\mathbb{Z}$ \cite{Ca1813}, we have $|A+B|\ge \min\{p, |A|+|B|-1\}=|A|+|B|-1.$

\end{proof}

\begin{lem}
If finite sets $A, B\subseteq\mathbb{Z}$ are nonempty sets, then $|A\widehat{+}B|\ge |A|+|B|-3.$

\end{lem}
\begin{proof}
 Similar with Lemma 1, since $A$ and $B$ are finite sets, define the element with the largest absolute value in $A$ as $a_x$ and the element with the largest absolute value in $B$ as $b_y$. Apparently, $\exists p$ such that $p$ is a prime number and  $p>\max\{a_m, b_n, |a_x|+|b_y| , |A|+|B|-3\}$. Thus, applying the general case of the Erdős-Heilbronn conjecture in $\mathbb{Z}/p\mathbb{Z}$ \cite{Na1995}, we have $|A\widehat{+}B|\ge \min\{p, |A|+|B|-3\}=|A|+|B|-3.$

\end{proof}

\begin{lem}$(A, B)$ is a critical pair and $|A|\ge 5$. If $B\subseteq A$, then $a_m=b_n$ and $(X, Y)$ is also a critical pair, where $X=A\setminus\{a_m\}$ and $Y=B\setminus\{a_n\}$.

\end{lem}

\begin{proof}
 Since $B\subseteq A$, we have $a_m\ge b_n$. Notice that $(A\widehat{+}\{b_1\})\cap(\{a_m\}\widehat{+}B)=\{a_m+b_1 \}$. Since $$(A\widehat{+}\{b_1\})\cup(\{a_m\}\widehat{+}B)\subseteq A\widehat{+}B,$$ we have  $$|A\widehat{+}B|\ge |(A\widehat{+}\{b_1\})\cup(\{a_m\}\widehat{+}B)|=|A\widehat{+}\{b_1\}|+|\{a_m\}\widehat{+}B|-1.$$ If $a_m>b_n$, then $a_m\notin B$. Thus, $$|A\widehat{+}\{b_1\}|+|\{a_m\}\widehat{+}B|-1=(|A|-1)+|B|-1=|A|+|B|-2.$$ A contradiction! Therefore, $a_m=b_n$.
\vspace{1pc}
 
Apparently, 
\begin{align*}
|(X\widehat{+}Y)\cap(\{a_m\}\widehat{+}A)| &= |(X\widehat{+}Y)\cap((\{a_m\}\widehat{+}A)\setminus\{a_{m-1}+a_m\})| \\
 &\le |(\{a_m\}\widehat{+}A)\setminus\{a_{m-1}+a_m\}| \\
 &= |A|-2.
\end{align*}
Now, because $B\subseteq A$, we have $A\widehat{+}B=(X\widehat{+}Y)\cup(\{a_m\}\widehat{+}A).$ Thus, according to Lemma 2,
\begin{align*}
|(X\widehat{+}Y)\cap(\{a_m\}\widehat{+}A)| &= |X\widehat{+}Y|+|\{a_m\}\widehat{+}A|-|(X\widehat{+}Y)\cup(\{a_m\}\widehat{+}A)| \\
 &= |X\widehat{+}Y|+|\{a_m\}\widehat{+}A|-|A\widehat{+}B| \\
 &\ge (|X|+|Y|-3)+(|A|-1)-(A|+|B|-3) \\
 &= |A|-3
\end{align*}
Thus, $|(X\widehat{+}Y)\cap(\{a_m\}\widehat{+}A)|\in \{|A|-2, |A|-3 \}.$
\vspace{1pc}
 
 If $|(X\widehat{+}Y)\cap(\{a_m\}\widehat{+}A)|=|A|-2$, then $$(X\widehat{+}Y)\cap(\{a_m\}\widehat{+}A)=(\{a_m\}\widehat{+}A)\setminus\{a_{m-1}+a_m\}.$$ This implies that for $\forall a_i\in X\setminus\{a_{m-1}\}\ (i\in\mathbb{Z}_+)$, $\exists x\in X$ and $\exists y\in Y$ with $x\ne y$ such that
 $a_m+a_i=x+y$. By taking $a_i=a_{m-2}$, we have $$a_m+a_{m-2}=x+y.$$ Since $a_m>\max\{x, y\}$, then we must have $x>a_{m-2}$ and $y>a_{m-2}$. Thus, $x=y=a_{m-1}$. A contradiction! Consequently, $|(X\widehat{+}Y)\cap(\{a_m\}\widehat{+}A)|=|A|-3.$
\vspace{1pc}

Through the following computation,
\begin{align*}
 |X\widehat{+}Y| &= |(X\widehat{+}Y)\cup(\{a_m\}\widehat{+}A)|-|\{a_m\}\widehat{+}A)|+|(X\widehat{+}Y)\cap(\{a_m\}\widehat{+}A)|  \\
 &= |A\widehat{+}B|-(|A|-1)+(|A|-3) \\
 &= (|A|+|B|-3)-(|A|-1)+(|A|-3) \\
 &= (|A|-1)+(|B|-1)-3 \\
&= |X|+|Y|-3,
\end{align*}
we complete the proof that $(X, Y)$ is a critical pair.

\end{proof}

\begin{lem}Let finite sets $A, B\subseteq \mathbb{Z}$ be nonempty sets. $(A, B)$ is a critical pair and $|A|\ge 5$. If $B\subseteq A$, then $(A, B)$ is a standard pair.

\end{lem}
\begin{proof}
Let's use induction on $|A|$ to prove Lemma 4.
\vspace{1pc}

When $|A|=5$, define $X=A\setminus\{a_5\}$ and $Y=B\setminus\{b_n\}.$ According to Lemma 3, $a_5=b_n$ and $(X,Y)$ is a critical pair. Since $|X|=4$, according to Theorem 3, we have $X=Y=\{a_1, a_2, a_3, a_4\}$ and $X\widehat{+}Y=\{a_1+a_4, a_2+a_4, a_3+a_4, a_1+a_2, a_1+a_3\}.$ Apparently, $A=B=\{a_1, a_2, a_3, a_4, a_5\}$. From the proof of Lemma 3, we know that 
$$|(X\widehat{+}Y)\cap({a_5}\widehat{+}A)|=|A|-3=2.$$ It's easy for us to observe that $a_2+a_4=a_5+a_1$ and $a_3+a_4=a_5+a_2$. By the proof of Theorem 3, we also have $a_2+a_3=a_1+a_4$. Consequently, $$a_5-a_4=a_4-a_3=a_3-a_2=a_2-a_1.$$ Thus, $(A, B)$
is a standard pair.
\vspace{1pc}

Now, suppose Lemma 4 is true for $|A|\le m-1\ (m\in\mathbb{Z}, m>5)$. Let's consider the case when $|A|=m$. Note $A'=A\setminus\{a_m\}$ and $B'=B\setminus\{b_n\}$, then $(A', B')$
is a critical pair and $a_m=b_n$ according to Lemma 3. Thus, by the induction hypothesis, $(A', B')$
is a standard pair. Without loss of generality, for some $\tau$, $d\in\mathbb{Z}$, we may assume that  $$A'=B'=\{\tau, \tau+d, \tau+2d, \cdots, \tau+(m-2)d\}$$ with $\tau<\tau+d<\tau+2d<\cdots< \tau+(m-2)d$. 
\vspace{1pc}

From the proof of Lemma 3, we know that $$|(A'\widehat{+}B')\cap({a_m}\widehat{+}A)|=|A|-3=|{a_m}\widehat{+}A|-2.$$ Thus, for $\forall u\in A'\setminus\{\tau+(m-3)d, \tau+(m-2)d\}$, $\exists v\in A', w\in B'$ such that $$a_m+u=v+w.$$ When $u=\tau+(m-4)d$, $\max\{v, w\}\ge \tau+(m-3)d$. Without loss of generality, assume that $v\ge w$. Thus, $v\in\{\tau+(m-3)d, \tau+(m-2)d\}.$
\vspace{1pc}

If $v=\tau+(m-3)d$, then $a_m=w+d$. From the definition of $a_m$, we know that $a_m>\tau+(m-2)d$. Because $w\le \tau+(m-2)d$, we must have $w=\tau+(m-2)d$. Consequently, $a_m=\tau+(m-1)d=b_n.$
\vspace{1pc}

If $v=\tau+(m-2)d$, then $a_m=w+2d$. According to the definition of $a_m$, we know that $a_m>\tau+(m-2)d$. Thus, since $w\le \tau+(m-2)d$, $w\in \{\tau+(m-3)d, \tau+(m-2)d\}$. Note that $w\ne v$, then there must be $w=\tau+(m-3)d.$ Therefore, $a_m=\tau+(m-1)d=b_n.$
\vspace{1pc}

As a result, we always have $$A=B=\{\tau, \tau+d, \tau+2d, \cdots, \tau+(m-2)d, \tau+(m-1)d\}$$ with $\tau<\tau+d<\tau+2d<\cdots<\tau+(m-2)d<\tau(m-1)d$. This completes our induction.
\end{proof}

\begin{lem}
Let finite sets $A, B\subseteq \mathbb{Z}$ be nonempty sets. If $|A|, |B|\ge 2$ and $B\subsetneqq A$, then $|A\widehat{+}B|\ge |A|+|B|-2.$
\end{lem}
\begin{proof}
Since $B\subsetneqq A$, according to Lemma 2, we have $$|A\widehat{+}B|\ge|A|+|B|-3. $$ According to Lemma 4, if $|A\widehat{+}B|=|A|+|B|-3$, then $A=B$. However, this contradicts with $B\subsetneqq A$. Thus, $$|A\widehat{+}B|>|A|+|B|-3.$$ Therefore, we have $|A\widehat{+}B|\ge|A|+|B|-2$.
\end{proof}

\begin{lem}
Let finite sets $A, B\subseteq \mathbb{Z}$ be nonempty sets. If $|A|, |B|\ge 2$ and $A\neq B$, then $|A\widehat{+}B|\ge |A|+|B|-2.$
\end{lem}
\begin{proof}
Using Lemma 5, we only need to consider the case when $A\neq B$ and $B\not\subset A$. Without loss of generality, assume $|A|\ge |B|$. If $|A\cap B|=0$, then according to Lemma 1, we have $$|A\widehat{+}B|=|A+B|\ge|A|+|B|-1>|A|+|B|-2. $$ If $|A\cap B|=1$, then note $X=A\cup B$. According to Lemma 1, we have $$|A\widehat{+}B|\ge |A\widehat{+}(B\setminus X)|=|A+(B\setminus X)|\ge |A|+|B\setminus X|-1\ge |A|+|B|-2.$$ Thus, we only need to consider the case when $|A\cup B|\ge 2$. Since $B\not\subset A$, we have $A\cap B\subsetneqq A\cup B$. Moreover, $A\cap B$, $A\cup B\in \mathbb{Z}$ and $|A\cup B|\ge 2$. Thus, according to Lemma 5, we have $$|(A\cup B)\widehat{+}(A\cap B)|\ge |A\cup B|+|A\cap B|-2=|A|+|B|-2.$$ Notice that for any $x\neq y$ with $x\in A\cup B$
and $y\in A\cap B$, if $x\in A$, then $y\in B$. Thus, $x+y\in A\widehat{+}B$. Similarly, if $x\in B$, then $y\in A$ and thus $x+y\in A\widehat{+}B$. As a result, $$(A\cup B)\widehat{+}(A\cap B)\subseteq A\widehat{+}B.$$ Therefore, $|A\widehat{+}B|\ge |(A\cup B)\widehat{+}(A\cap B)|\ge |A|+|B|-2$.

\end{proof}
Back to the main problem, if $(A, B)$ is a critical pair, Lemma 6 indicates that $A=B$. Thus, we have $B\subseteq A$. According to Lemma 4, $(A, B)$ is a standard pair when $|A|\ge 5$. This finishes our proof for the sufficiency of the conditions. 
\end{proof}
\vspace{1pc}

\section{Partial Results in $\mathbb{Z}/p\mathbb{Z}$}

\begin{thm}
Let sets $A, B\subseteq \mathbb{Z}/p\mathbb{Z}$ be nonempty sets such that $p\ge |A|+|B|-2$ and $a_m+b_n<p$. Then $|A\widehat{+}B|=|A|+|B|-3$ if and only if $A=B$ and one of the following holds:
\begin{enumerate}
    \item $|A|=2$ or $|A|=3$;
    \item $|A|=4$ and $A=\{a, a+d, c, c+d\}$;
    \item $|A|\ge 5$ and $A$ is an arithmetic progression.
\end{enumerate}
\end{thm}

\begin{proof}

Since $a_m+b_n<p$, we have $a+b<p$ for $\forall a\in A, b\in B$. This leads to 
$$\{a+b: a\in A, b\in B, a\ne b\}=\{a+b \mod{p}: a\in A, b\in B, a\ne b\}.$$
This equation here indicates that the proof is equivalent to the case in $\mathbb{Z}$, which we have completed.

\end{proof}

\begin{thm}
Let sets $A, B\subseteq \mathbb{Z}/p\mathbb{Z}$ be nonempty sets such that $p\ge|A|+|B|$, where $p$ is a prime. Suppose $|A|\ge 5$ and $(A, B)$ is a critical pair. Given $d$ as an arbitrary generator of $\mathbb{Z}/p\mathbb{Z}$, then length of the longest gap in $B$
with respect to $d$ is not less than $|A|$.
\end{thm}

\begin{proof}

Let $A=\{\tau+id: i=0, 1,\cdots, |A|-1\}$ and $B=\{x_id: i=0, 1,\cdots, |B|-1\}$, where $\tau\in \mathbb{Z}/p\mathbb{Z}$ and $x_i\in \mathbb{Z}_+\cup\{0\}$. Define the length of the longest gap in $B$ with respect to $d$ as $t$ and suppose $t\le |A|-1$.
\vspace{1pc}

Without loss of generality, assume $\tau=0$, or else use $$A'=\{id: i=0, 1,\cdots, |A|-1\}$$ and $$B'=\{x_id-\tau: i=0, 1,\cdots, |B|-1\}$$ to represent $A$ and $B$. Since $0\in A$, apparently $B\setminus\{0\}\subseteq A\widehat{+}B$.
\vspace{1pc}

Define sets $B_i\subseteq B$ as the following:
$$B_i=\{m_id, (m_i+1)d, \cdots, \widehat{m_i}d\}\ (m_i\le \widehat{m_i}), $$
where $m_{i+1}-\widehat{m_i}\ge 1$. By definition, we can get a partition of $B$, where  $I$ a finite set of positive integers and $B=\sqcup_{i\in I} B_i$. Note $\mathfrak{B}_1=\{B_i: \widehat{m_i}>t\}$ and $\mathfrak{B}_2=\{B_i: \widehat{m_i}\le t\}=\{B_1, B_2, \cdots, B_k\}\ (1\le k\le |B|)$.
\vspace{1pc}

We first claim that for $\forall g\in \mathbb{Z}_+\cup\{0\}$, $gd\in A\widehat{+}B$ if there $\exists j\in I$ such that $\widehat{m_j}<g<m_{j+1}$ and $\widehat{m_j}>\widehat{m_k}$.
\vspace{1pc}

In particular, if $0\le g<m_1$ then $gd\in A\widehat{+}B$. Observing that $g-\widehat{m_j}<m_{j+1}-\widehat{m_j}\le t+1\le |A|,$ we have $g-\widehat{m_j}\le |A|-1$ and thus $(g-\widehat{m_j})d\in A$. Since $\widehat{m_j}>\widehat{m_k}$, $\widehat{m_j}>t\ge g-\widehat{m_j}.$ Then $gd=(g-\widehat{m_j})d+\widehat{m_j}d\in A\widehat{+}B.$
\vspace{1pc}

We now proof that for $\forall xd\in A$, if $x\ne 2\widehat{m_i}\ (i=1, 2,\cdots, k)$, then $xd\in A\widehat{+}B.$
\vspace{1pc}

To prove this claim, assume that $\widehat{m_i}<x<m_{i+1}$ for some $i\in I$. Again $$xd=(x-\widehat{m_i})d+\widehat{m_i}d\in A\widehat{+}B,$$ unless $(x-\widehat{m_i})d=\widehat{m_i}d$, which indicates $x=2\widehat{m_i}$.
\vspace{1pc}

For $i=1, 2, \cdots, k-1$, if $|B_i|\ge 2$, then $2\widehat{m_i}d\in A\widehat{+}B$, because $(\widehat{m_i}-1)d\in B$ and $2\widehat{m_i}d=(\widehat{m_i}+1)d+(\widehat{m_i}-1)d$.
\vspace{1pc}

For $i=2, 3 \cdots, k-1$, if there exists $\widehat{m_{i-1}}$ and $m_{i+1}$ such that $\widehat{m_{j-1}}<\widehat{m_i}<m_{i+1}$, then we claim that $2\widehat{m_i}d\in A\widehat{+}B.$ 
\vspace{1pc}

Note that $0\le \widehat{m_{i-1}}<\widehat{m_i}<m_{i+1}$. If $2\widehat{m_i}<m_{i+1}$, then $$m_{i+1}-\widehat{m_i}>\widehat{m_i}\ge \widehat{m_i}-\widehat{m_{i-1}}.$$ So, $2\widehat{m_i}-\widehat{m_{i-1}}<m_{i+1}\le \widehat{m_{i+1}}\le t\le |A|-1$. Thus, $(2\widehat{m_i}-\widehat{m_{i-1}})d\in A.$ This further leads to $2\widehat{m_i}d=(2\widehat{m_i}-\widehat{m_{i-1}})d+\widehat{m_{i-1}}d\in A\widehat{+}B.$ Similarly, if $2\widehat{m_i}\ge m_{i+1}$, then we have $2\widehat{m_i}-m_{i+1}\in A$. Thus,
$2\widehat{m_i}d=(2\widehat{m_i}-\widehat{m_{i+1}})d+\widehat{m_{i+1}}d\in A\widehat{+}B.$
\vspace{1pc}

Combining all these claims, we find the only possible elements in $\mathbb{Z}/p\mathbb{Z}$
that are not in $A\widehat{+}B$:
\begin{enumerate}
    \item $\widehat{m_1}$, if $\widehat{m_1}=0$;
    \vspace{1pc}
    
    \item $2\widehat{m_1}$;
    \vspace{1pc}
    
    \item $2\widehat{m_k}.$
\end{enumerate}
Apparently, the first two cases can not happen together at the same time unless $\widehat{m_1}=2\widehat{m_1}=0.$ So, at most two elements in $\mathbb{Z}/p\mathbb{Z}$  are not in  $A\widehat{+}B.$ This gives $|A\widehat{+}B|\ge p-2\ge |A|+|B|-2.$ A contradiction! Thus, we must have $t\ge |A|.$

\end{proof}

\begin{thm}
Let sets $A, B\subseteq \mathbb{Z}/p\mathbb{Z}$ be nonempty sets such that $p\ge|A|+|B|$, where $p$ is a prime. Suppose $|A|\ge 5$ and $(A, B)$ is a critical pair. If $A$
is a standard set, then $(A, B)$ is a standard pair.
\end{thm}

\begin{proof}

Define $A$ and $B$ in the same way as in Theorem 6. According to Theorem 6, the length of the largest gap in $B$ with respect to $d$ is not less than $|A|$. Denote the largest gap as $[u, v]_d$, then $(u-1)d\in B$. $A=\{id: i=0, 1,\cdots, |A|-1\}$ gives $$\{(u-1)d\}\widehat{+}A=\{(u-1+i)d: i=0, 1,\cdots, |A|-1\}\subseteq A\widehat{+}B.$$ Since $0\in A$, $\{0\}\widehat{+}B\subseteq A\widehat{+}B$. According to the definition of $u$, we know that $$(\{(u-1)d\}\widehat{+}A)\cap (\{0\}\widehat{+}B)=\{(u-1)d\}.$$
Observing that 
\begin{align*}
 |A\widehat{+}B| &= |A|+|B|-3  \\
 &\ge |(\{(u-1)d\}\widehat{+}A)\cup (\{0\}\widehat{+}B)| \\
 &= |(\{(u-1)d\}\widehat{+}A)|+|(\{0\}\widehat{+}B)|-1 \\
 &\ge (|A|-1)+(|B|-1)-1 \\
 &= |A|+|B|-3,
\end{align*}
all equalities must hold. Therefore, $\{(u-1)d\}\in A$, $0\in B$ and 
\begin{equation}\label{eqn51}
A\hat{+}B=(\{(u-1)d\}\widehat{+}A)\cup (\{0\}\widehat{+}B). 
\end{equation}

 Define $X=(\{(u-1)d\}\widehat{+}A)\cup (\{0\}\widehat{+}B)$. We can prove that  $(w+1)d\in B$ for $\forall wd\in B$ with $w\ne 1$ and $w\ne (u-1).$ If not, $(w+1)d\in (A\widehat{+}B)\setminus X,$ which is not possible. Applying this argument, we get $d\in B$, for $0\in B$. Similarly, consider $0+2d$, then $0+2d=2d\in B$. So we have $xd\in B$ for $\forall x\notin [u, v]_d$ by repeating the previous argument. In particular, $B$ is a standard set and $B=\{b+td: t=0, 1, \cdots, |B|-1\}$.
\vspace{1pc}

Now, we can apply a similar argument as \cite{Su2013} to prove that $(A, B)$ is a standard pair. We have $A\widehat{+}B=\{sd+(b+td): 0\le s\le |A|-1, 0\le t\le |B|-1, sd\ne b+td\}$. Since $sd+(b+td)=(s\pm 1)d+b+(t\mp 1)d$, then even if $sd=b+td$, we have the sum be written into the sum of two distinct elements from $A$ and $B$ unless $s=t=0$ or $s=|A|-1$ and $t=|B|-1$. Therefore, $A\widehat{+}B=A+B$ unless $b=0$ or $(|A|-1)d=b+(|B|-1)d$. By putting $b'=b+(|B|-1)d$ and forming the arithmetic progression with the common difference $d'=-d$, we can reduce the second case to the first one. So, $A\widehat{+}B=A+B$ unless $b=0$. Since $A+B\ge \min\{p, |A|+|B|-1\}>|A|+|B|-3$, we must have $b=0$.
\vspace{1pc}

Thus, we can write $$A\widehat{+}B=\{sd+td\mod{p}: sd\in A, td\in B, sd\ne td\}.$$ If $sd=td$ for $s\ne t$ (say without the loss of generality that $s>t$), then $$(s-t)d\equiv 0\mod{p}.$$ Because $d\not\equiv 0\mod{p}$, $p\mid (s-t)$. By our definition for $p$, $$|A|\ge (s-t)\ge p>|A|+|B|>|A|,$$ which is a contradiction. We now must conclude that $sd=td$ implies $s=t$.
\vspace{1pc}

Again, for the case where $sd=td$, we can write $$sd+td=(s\pm 1)d+(t\mp 1)d, $$ unless $s, t=0$ or $s=t=|A|-1=|B|-1$. However, if $|A|\ne |B|$, then we only get the case $s=t=0$, which means that $$ |A\widehat{+}B|\ge |A+(B\setminus\{0\})|\ge |A|+(|B|-1)-1=|A|+|B|-2,$$ which is also a contradiction. Therefore, we must have $|A|=|B|$ and this completes the proof.

\end{proof}

\section{Concluding Remarks \& Further Thoughts}
We have given an elegant and completely elementary proof for the inverse Erdős-Heilbronn problem in $\mathbb{Z}$ and give some partial results on the elementary proof for the inverse Erdős-Heilbronn problem in $\mathbb{Z}/p\mathbb{Z}$. Our strategies in $\mathbb{Z}$ need to be further improved to apply to $\mathbb{Z}/p\mathbb{Z}$, for we are considering the equations under the meaning of module $p$ in $\mathbb{Z}/p\mathbb{Z}$. 
\vspace{1pc}

For Theorem 5, we add the technical condition of $a_m+b_n<p$ in order to apply the theorems that we have obtained in $\mathbb{Z}$ to $\mathbb{Z}/p\mathbb{Z}$. For Theorem 6 and Theorem 7, we drop this technical condition and ask $p\ge |A|+|B|$ instead of $p\ge |A|+|B|-2$ in the case of $\mathbb{Z}$ to complete the proof of our results. 

\vspace{1pc}
Through Theorem 7, we improve Theorem 5.1 in \cite{Su2013}, which gives an elementary proof of the inverse Erdős-Heilbronn problem in $\mathbb{Z}/n\mathbb{Z}$ under the condition that elements in sets $A$ and $B$ for arithmetic progressions with the same common difference, by giving an elementary proof in $\mathbb{Z}/p\mathbb{Z}$ that only requires $p\ge |A|+|B|$ and $A$ to be a standard set.
\vspace{1pc}

Moreover, we may very likely guess that an elementary proof in $\mathbb{Z}/p\mathbb{Z}$ can be achieved through an analogy of Vosper's elementary proof \cite{Vo1956} on the inverse Cauchy-Davenport problem.
\vspace{1pc}

In addition, through of elementary proof in the case of $\mathbb{Z}$, we give an explicit description of $A\widehat{+}B$ as the union of $(a_{max}\widehat{+}B)\cup (A\widehat{+}b_{min})$, so we may expect in general, when $|A\widehat{+}B|$ is close to $|A|+|B|$ in $\mathbb{Z}$, $A\widehat{+}B$ could almost be written as a union of such copies. Similar results may very possibly be found in $\mathbb{Z}/p\mathbb{Z}$ if a completely elementary proof is found as hinted by our results in $\mathbb{Z}/p\mathbb{Z}$.


\begin{thebibliography}{99}
	


\bibitem{Ca1813}
A.L. Cauchy, \textit{Recherches sur les nombres}, J. $\acute{\mathrm{E}}$cole polytech, \textbf{9} (1813) 99–116.

\bibitem{Da1935}
H. Davenport, \textit{On the addition of residue classes}, Journal of the London Mathematical Society, \textbf{10} (1935) 30-32.

\bibitem{Eh1964}
P. Erdős and H. Heilbronn, \textit{On the addition of residue classes modulo $p$}, Acta Arith., \textbf{9} (1964) 149-159.

\bibitem{Jy1994}
J.A. Dias sa Silva and Y.O. Hamidoune, \textit{Cyclic spaces for Grassmann derivatives and additive theory}, Bull. London Math. Soc., \textbf{26} (1994) 140-146.

\bibitem{Na1995}
N. Alon, M.B. Nathanson, and I. Ruzsa, \textit{Adding distinct congruence classes modulo a prime}, American Mathematical Monthly, \textbf{102} (1995) 250-255.

\bibitem{Gy2009}
Gy. Károlyi, \textit{Restricted set addition: The exceptional case of the Erdős-Heilbronn conjecture}, Journal of Combinatorial Theory A, \textbf{116} (2009) 741-746.

\bibitem{Su2013}
S.M. Jayasuriya, S.D. Reich, and J.P. Wheeler, \textit{On the inverse Erdős-Heilbronn problem for restricted set addition in finite groups}, \emph{https://arxiv.org/pdf/1210.6509.pdf}.

\bibitem{GTM165}
M.B. Nathanson, \textit{Additive Number Theory: Inverse Problems and the Geometry of Sumsets}, Graduate Texts in Mathematics 165, Springer (1996).
	
\bibitem{Vo1956}
A.G. Vosper, \textit{The Critical Pairs of Subsets of a Group of Prime Order}, Journal of the London Mathematical Society, \textbf{31} (1956) 200-205.
	
\end{thebibliography}
\end{document}